\documentclass[11pt,a4paper]{article}
\usepackage{latexsym}
\usepackage{amssymb}
\usepackage{amsthm}
\usepackage{xcolor}
\usepackage{multicol}
\usepackage[leqno]{amsmath}
\title{The $C^0$ estimate for the quaternionic Calabi conjecture}
\author{Marcin Sroka}
\date{}

\usepackage{geometry}
\newtheorem*{theorem*}{Theorem}
\newtheorem*{conjecture*}{Conjecture}
\newtheorem{theoremletter}{Theorem}

\newtheorem{lemma}{Lemma}
\newtheorem*{lemma*}{Lemma}
\newtheorem*{sublemma*}{Sublemma}

\newtheorem{definition}{Definition}
\newtheorem*{definition*}{Definition}
\newtheorem{remark}{Remark}
\newcommand{\hh}{\mathbb{H}}
\newcommand{\rr}{\mathbb{R}}
\newcommand{\cc}{\mathbb{C}}

\newcommand{\ii}{\mathfrak{i}}
\newcommand{\jj}{\mathfrak{j}}
\newcommand{\kk}{\mathfrak{k}}

\newcommand{\jpar}{\partial_J}

\begin{document}
\maketitle
\textbf{Abstract:} We prove the $C^0$ estimate for the quaternionic Monge-Amp\`ere equation on compact hyperK\"ahler with torsion manifolds. Our goal is to provide a simpler proof than the one presented in \cite{AS17}.

%\textbf{MSC2010:} 


\section{Introduction and preliminaries}

The subject of this note is the quaternionic Monge-Amp\`ere equation on a compact hyperK\"ahler with torsion, later abbreviated as HKT, manifold. 

We start by briefly reminding what are HKT manifolds. Those belong to the realm of quaternionic geometries and emerged from mathematical physics as the internal space of certain super-symmetric sigma models. The established reference for a mathematical treatment is \cite{GP00} which we follow below. Let us recall that a hypercomplex manifold is one, say $M$, equipped with three complex structures $I$, $J$ and $K$ satisfying the quaternionic relation 
\begin{center} $I\circ J \circ K= -id_{TM}$. \end{center} 
A very important note here is that for us endomorphisms act from the right on the tangent space. This convention is compatible with the one taken up by Alesker, Shelukhin and Verbitsky in their papers on the quaternionic Calabi conjecture. In that case each tangent space $T_x M$, for $x \in M$, becomes a right $\hh$ module, or a vector space as is accepted to say, where multiplication by $\ii$, $\jj$ and $\kk$ is given by $I_x$, $J_x$ and $K_x$ respectively. Now, $(M,I,J,K,g)$ is called hyperhermitian if $g$ is a Riemannian metric which is hermitian with respect to $I$, $J$ and $K$ i.e. 
\begin{center} $g=g(\cdot I,\cdot I)=g(\cdot J,\cdot J)=g(\cdot K,\cdot K)$.\end{center} 
A hypercomplex manifold admits the whole sphere of complex structures namely 
\begin{center} $S_M= \{aI+bJ+cK \: | \: a^2+b^2+c^2=1\}$ \end{center} 
and a hyperhermitian metric $g$ is hermitian with respect to all of them. For a given $L \in S_M$ we denote the associated hermitian form by $\omega_L$ i.e. $\omega_L=g(\cdot L, \cdot)$.         
\begin{definition}
A hyperhermitian manifold $(M,I,J,K,g)$ is called HKT if $$\partial \Omega=0$$ where $\Omega := \omega_J - \ii \omega_K$ and $\partial$ in the whole paper is taken with respect to $I$.
\end{definition}
\begin{remark}
The form $\Omega$ is called an HKT form associated to an HKT metric $g$ and it is of type $(2,0)$ with respect to $I$. In \cite{GP00} the definition of an HKT manifold is different. There it is a hyperhermitian manifold for which a linear connection preserving $g$, $I$, $J$, $K$ and having a skew-symmetric torsion tensor exists. This is equivalent to the equality of the three Bismut connections for hermitian manifolds $(M,I,g)$, $(M,J,g)$ and $(M,K,g)$ respectively. These conditions are equivalent to our definition as shown in Proposition 2 of \cite{GP00}. Let us note that, as an easy calculation shows, the condition $d \Omega = 0$ corresponds to $M$ being hyperK\"ahler thus HKT manifolds constitute an intermediate class between hyperhermitian and hyperK\"ahler manifolds.   
\end{remark}

The so called quaternionic Monge-Amp\`ere equation in a compact setting was introduced by Alesker and Verbitsky in \cite{AV10}, it is strongly motivated by its complex analogue. In order to explain it properly we need to elaborate a little more on the geometry of hypercomplex manifolds. First of all, there is a quaternionic analog of the Dolbeault differential operator $\overline{\partial}$ obtained in the following way. Given any field of endomorphisms $L$ on $TM$, acting according to our convention from the right, we define its left action on the space of complex valued, smooth differential forms by 
\begin{center}
$L: \Lambda^k_\cc (M) \ni \alpha \longmapsto \alpha(\cdot L,..., \cdot L) \in \Lambda^k_\cc (M)$.
\end{center}      
The reader sees we use the same symbol, here $\Lambda^k_\cc (M)$, for the vector bundle and the space of its smooth sections. In the case of hypercomplex manifolds we thus obtain the right action of $Sp(1)$ on $TM$ and the left one on $\Lambda^k_\cc (M)$ for any $k$. The twisted Dolbeault differential operator was introduced in \cite{V02} as
\begin{center}
$\jpar:=J^{-1} \circ \overline{\partial} \circ J: \Lambda^k_\cc (M) \rightarrow \Lambda^{k+1}_\cc (M)$
\end{center}
where $\overline{\partial}$ is again everywhere assumed to be taken with respect to $I$. One may check that 
\begin{center}
$\jpar: \Lambda^{p,q}_I (M) \rightarrow \Lambda^{p+1,q}_I (M)$ \\
$\partial \jpar + \jpar \partial = 0 $
\end{center}
where $\Lambda^{p,q}_I (M)$ is the space of differential forms of type $(p,q)$ with respect to $I$. It was observed in \cite{V02} that, from the formal point of view, the pair $\partial$, $\jpar$ is similar to $\partial$, $\overline{\partial}$. This analogy can be pushed further. Since $I$ and $J$ anti-commute the action of $J$, on the forms of a pure type with respect to $I$, is
\begin{center}
$J: \Lambda^{p,q}_I (M) \rightarrow \Lambda^{q,p}_I (M)$
\end{center} and consequently $J$ composed with the bar operator is an involution on $\Lambda^{p,q}_I(M)$ if $p+q$ is even. In \cite{AV06} a subboundle of fixed points for this endomorphism in $\Lambda^{2k,0}_I (M)$ was denoted by $\Lambda^{2k,0}_{I,\rr} (M)$ i.e. $\alpha \in \Lambda^{2k,0}_{I,\rr} (M) $ iff $J \alpha = \overline{\alpha}$ and such a from is called q-real. Furthermore the notion of q-positivity is introduced there as well. For us it will be essential that $\Omega$ is a q-positive form and in general a $(2,0)$ form $\alpha$ is q-positive if $\alpha(X,XJ) \geq 0$, or equivalently $\alpha(Z,\overline{Z}J) \geq 0$, for any $X$ a real vector field and $Z$ a $(1,0)$ vector field. We refer to Section 2 of \cite{AV06} for more details on positivity in order to avoid unnecessary redundancy. On a given HKT manifold of the quaternionic dimension $n$ the bundle $\Lambda^{2n,0}_{I,\rr} (M)$ is trivial as its trivialization is given by $\Omega^n$. Motivated by one of the equivalent formulations of the Calabi conjecture Alesker and Verbitsky posted its version for HKT manifolds, cf. \cite{AV10}. 
\begin{conjecture*} Given any, necessarily q-positive, section of $\Lambda^{2n,0}_{I,\rr} (M)$ i.e. a section of the form $e^F \Omega^n$ for some $F \in C^\infty(M)$ there exists an HKT metric $g$ on $(M,I,J,K)$ such that the associated HKT form is $\Omega + \partial \jpar \phi$ for some $\phi \in C^\infty(M)$ and it satisfies $\left( \Omega + \partial \jpar \phi \right)^n=Ae^F \Omega^n$ for some $A>0$.
\end{conjecture*} 
\noindent As was noted in \cite{AV06} the form $\Omega + \partial \jpar \phi$ comes from an HKT metric provided it is q-positive so the above conjecture is equivalent to solvability of the equation
\begin{center}
\[ \tag{1.0}\label{1.0} \begin{cases}
\left( \Omega + \partial \jpar \phi \right)^n=Ae^F \Omega^n \\
\Omega + \partial \jpar \phi \geq 0.
\end{cases} \]
\end{center}
\begin{remark}
Originally the conjecture was posted in \cite{AV10} assuming in addition that the canonical bundle $\Lambda^{2n,0}_I(M)$ of $(M,I)$ is trivial holomorphically. It is always trivial topologically as $\Omega^n$ gives the trivialization but in general this section is not holomorphic. Later, in \cite{AS13, AS17}, it was stated in the form as above.   
\end{remark}
\begin{remark}
The question arises, like in the case of the complex Monge-Amp\`ere equation on hermitian manifolds, why to look for a metric whose associated HKT form is a $\partial \jpar \phi$ perturbation of the original one. This does not follow from a simple requirement of belonging to the de Rham class $[\Omega]_{dR}$ since in general the $\partial \jpar$ lemma is not true on a given HKT manifold. It is true though for example for hyperK\"ahler or, more generally, $Sl_n(\hh)$ manifolds, cf. \cite{GLV17}. Being a $\partial \jpar \phi$ perturbation of $\Omega$ becomes necessary if one agrees to look for solutions belonging to the class of $\Omega$ in a Bott-Chern type cohomology group $$H^{2,0}_{BC}(M) := \frac {\{\eta \in \Lambda^{2,0}_I(M) \: | \: \partial \eta=\jpar \eta = 0\}}{\partial \jpar C^\infty(M)}$$ discussed in \cite{GLV17}.
\end{remark}
\begin{remark}
Provided the canonical bundle $\Lambda^{2n,0}_I(M)$ is trivial holomorphically the necessary condition for solvability of (\ref{1.0}) is $$\int\limits_{M}(1-Ae^F) \Omega^n \wedge \overline{\theta}=0,$$ where $\theta$ gives the holomorphic trivialization. This can be seen from Stokes' theorem. When the canonical bundle is non-trivial any holomorphic section, assuming it is not a zero section, gives rise to the condition as above. We do not know whether there are examples of HKT manifolds for which the space of holomorphic sections is at least two dimensional, certainly there are examples with no sections at all like quaternionic Hopf manifolds. It is not clear for us whether the conditions we obtain in that case are, in general, different or not. This observation was drawn to our attention by S. Dinew.
\end{remark}

Let us now give an overview of the advances towards proving the conjecture. The strategy is, of course, to use the continuity method for which a priori estimates are crucial. It is possible to obtain the $C^0$ estimate in the case when the canonical bundle is trivial by repeating the Moser iteration method used by Yau in \cite{Y78}, this was done by Alesker and Verbitsky in \cite{AV10}. In \cite{AS13} this bound was shown to hold when the hypercomplex structure is locally flat by using the method of B\l ocki from \cite{B05}. We owe a word of explanation for non experts what a locally flat structure means. By definition a complex structure is an integrable $GL_n(\cc)$ structure i.e. any complex manifold locally looks like $\cc^n$. This is not the case for hypercomplex structures which are known to be just 0-integrable $Gl_n(\hh)$ structures and, in general, are not integrable in a strong sense i.e. locally $I$, $J$ and $K$ are not pull backs of the standard hypercomplex structure induced by $\ii$, $\jj$ and $\kk$ in $\hh^n$. When the last condition is true the hypercomplex structure is said to be locally flat and such structures were studied originally in \cite{S75}. 
Under an even stronger assumption that the HKT manifold is a flat hyperK\"ahler one the conjecture was proven by Alesker in \cite{A13}. The assumption that the hyperK\"ahler metric is flat, in the sense that the full Riemann curvature tensor vanishes, implies in particular that the hypercomplex structure is flat. Actually the manifold is then a finite cover of a torus by Bieberbach's theorem on compact, flat Riemannian manifolds.
One of the main difficulties in repeating B\l ocki's argument in the general case is non-integrability of a hypercomplex structure. This prevents the problem from being automatically transferred to the domain in $\hh^n$. That issue was addressed by Alesker and Shelukhin in \cite{AS17}. They provided the proof of the $C^0$ estimate for the general case, i.e. without any additional assumption on an HKT structure, following the scheme of \cite{B05}. It turned out though that the proof of one technical fact needed for the reasoning, Theorem 3.2.2 in \cite{AS17}, is surprisingly complicated and occupies a central part of that paper. We intend to give another, in our opinion simpler, proof of the $C^0$ estimate for the equation (\ref{1.0}) i.e. of the theorem below.
\begin{theoremletter}\label{main}
Let $(M^n,I,J,K,g)$ be a compact HKT manifold and $F \in C^\infty(M)$. There exists a constant $C$, depending only on the HKT structure, $q>2n$ and $\parallel e^{F}\parallel_{L^q}$ (in particular only on $\parallel e^{F} \parallel_{L^\infty}$ and this depends only on $\sup\limits_M F$), such that for any smooth solution $\phi$ of the quaternionic Monge-Amp\`ere equation \begin{center} \[ \tag{1.1}\label{1.1}  \begin{cases}(\Omega + \partial \jpar \phi)^n= e^F\Omega^n \\ \Omega + \partial \jpar \phi \geq 0 \\ \sup\limits_M \phi = 0 \end{cases}  \] \end{center} the bound $\parallel \phi \parallel_{L^\infty} \leq C$ holds.
\end{theoremletter}
The proof we present is strongly motivated by the reasoning performed in \cite{TW10b} which is a refined version of the one described in \cite{TW10a}. This in turn is based on an inequality obtained originally by Cherrier in \cite{Ch87}. The method emerged in the course of proving the $C^0$ estimate for the complex Monge-Amp\`ere equation on a compact hermitian, implicitly non-k\"ahler, manifold. The general strategy we take is as follows. Firstly we prove the so called Cherrier type inequality, Lemma \ref{Cherrier}, for the assumed solution of (\ref{1.1}). Then using the Moser iteration method we obtain a special bound on $\inf\limits_M \phi$, Lemma \ref{prep1}, but still not being the desired estimate since the right hand side depends on $\phi$. From purely measure theoretic reasons this shows that the values of $\phi$ are separated from $\inf\limits_M \phi$ by a positive constant, independent of $\phi$ as it turns out, on a set of a positive, independent of $\phi$, measure, see Lemma \ref{prep2}. From this one can see the uniform bound follows easily provided we have at least an $L^1$ a priori estimate for which we refer to \cite{AS13} where it was proven via the bounded Green function argument. The $L^1$ bound from \cite{AS13} is also needed in \cite{AS17}, cf. Step 1 of the proof of Theorem 1.1.13, so there is no sweeping the issue under the carpet here.       

\textbf{Acknowledgments:} I would like to thank my supervisor S\l awomir Ko\l odziej for a constant help, reading the manuscript and the time he has spared for me. I have to mention numerous discussions with S\l awomir Dinew on complex and quaternionic Monge-Amp\`ere equations for which I am very grateful. I thank the referee for pointing out that the dependence of $C$ in Theorem \ref{main} can be relaxed from the $L^{\infty}$ to an $L^q$ norm of the RHS. This research was partially supported by the National Science Center of Poland grant number 2017/27/B/ST1/01145. 

\section{The $C^0$ estimate for the equation (\ref{1.1})}

In this note we apply the convention that unless explicitly stated any constant, not written of what it is dependent, is independent of $\phi$. When we want to express on what the constant is dependent we put those quantities in brackets for example $C \left(p,\parallel f \parallel_{L^\infty(M)} \right)$.  The same letter may denote different constants from line to line just to avoid unnecessary indexing. All the $L^q$ norms are taken w.r.t the volume element $\left(\Omega \wedge \overline{\Omega} \right)^n$.  

In Subsection 2.1 we prove the Cherrier type inequality, cf. (22) in \cite{Ch87} or Lemma 2.1 in \cite{TW10b}, which is a cornerstone of the reasoning. The proof of Theorem \ref{main} is finished in Subsection 2.2.

\subsection{A Cherrier type inequality for the quaternionic Monge-Amp\`ere equation}

\begin{lemma}\label{Cherrier} There exist positive constants $C$, $p_0$ both depending on the HKT geometry of the manifold, $q>2n$ and $\parallel e^{F}\parallel_{L^q}$ such that for any solution of (\ref{1.1}), $r$ being H\"older's conjugate of $q$ and any $p \geq p_0$ \begin{center} $\int\limits_{M} |\partial e^{-\frac p 2 \phi }|^{2}_{g} \left(\Omega \wedge \overline{\Omega} \right)^n \leq C p \parallel e^{- \phi} \parallel_{L^{pr}}^p$. \end{center}  \end{lemma}

\begin{proof}[Proof of Lemma \ref{Cherrier}] 
Let $\Omega_\phi := \Omega + \partial \jpar \phi$. Using the Stokes theorem we obtain that for any $p>0$ and fixed $q>2n$

\begin{center} \begin{equation} \tag{2.1} \label{2.1} \begin{split} 
& C \left(\parallel e^F \parallel_{L^q} \right) \parallel e^{-\phi} \parallel_{L^{pr}}^p \geq \int\limits_{M} e^{-p\phi} (e^F-1) \Omega^n \wedge \overline{\Omega^n} \\
& = \int\limits_{M}e^{-p\phi}(\Omega_{\phi}^n - \Omega^n) \wedge \overline{\Omega^n} = \int\limits_{M}e^{-p\phi} \partial \jpar \phi \wedge \alpha \wedge \overline{\Omega^n} \\
& = p \int\limits_{M}e^{-p\phi} \partial \phi \wedge \jpar \phi \wedge \alpha \wedge \overline{\Omega^n} + \int\limits_{M}e^{-p\phi} \jpar \phi \wedge \beta \wedge \alpha \wedge \overline{\Omega^n} \end{split} \end{equation} \end{center} 
where $r$ is H\"older's conjugate of $q$, $\alpha= \sum\limits_{k=0}^{n-1} \Omega_{\phi}^k \wedge \Omega^{n-1-k}$ and $\partial \left( \overline{\Omega^n} \right)=\beta \wedge \overline{\Omega^n}$ for some $(1,0)$ form $\beta$. First inequality above is the H\"older inequality and in the integration by parts we used $\partial \alpha = 0$. 

Our goal is to estimate the second factor on the right hand side of (\ref{2.1}) which we reduce to finding a uniform pointwise bound on \begin{center} $\jpar \phi \wedge \beta \wedge \alpha \wedge \overline{\Omega^n}$.\end{center} 
This follows from the analogue of the inequality $(2.2)$ from \cite{TW10b} as in the lemma below.
\begin{lemma} \label{pointwise} There exists a positive constant $B$ depending on $\beta$ such that 
\begin{center} \[ \tag{2.2} \label{2.2} \left| \frac{\jpar \phi \wedge \beta \wedge \Omega_{\phi}^k \wedge \Omega^{n-1-k}}{{\Omega^n}} \right| \leq \frac B \epsilon \frac{ \partial \phi \wedge \jpar \phi \wedge \Omega_{\phi}^k \wedge \Omega^{n-1-k}}{\Omega^n} + B \epsilon \frac{\Omega_{\phi}^k \wedge \Omega^{n-k}}{\Omega^n} \] \end{center} 
for any $\epsilon>0$ and $k\in \{0,...,n-1\}$. \end{lemma}

\begin{proof}[Proof of Lemma \ref{pointwise}] 
Let us note that, like in the complex case, one is able to simultaneously diagonalize, in a certain sense, both $\Omega$ and $\Omega_\phi$. Precisely we claim that for each $x \in M$ there exists a basis of $T_x^{1,0}M$, decomposition with respect to $I$, of the form $e_1,(\overline{e_1})J,...,e_n,(\overline{e_n})J$ such that 
\begin{center} $\Omega(e_i,e_j)=\Omega_\phi(e_i,e_j)=\Omega \left(e_i,(\overline{e_j})J \right)=\Omega_\phi \left(e_i,(\overline{e_j})J \right)=0$ for $i \not = j$.\end{center}
This follows from Lemma \ref{diagon} below by taking $\Omega_1=\Omega$ and $\Omega_2=\Omega_\phi$.

\begin{lemma} \label{diagon}
Let $\Omega_1$ be a strictly positive $(2,0)$ form, i.e. $\Omega_1(z,\overline{z}J)>0$ for any non zero $(1,0)$ vector $z$, and $\Omega_2$ a q-real $(2,0)$ form on $M$. For each $x\in M$ there exists a basis $e_1$, $(\overline{e_1})J$, ...,$e_n$, $(\overline{e_n})J$ of $T^{1,0}_x M$ such that 
\begin{center} \[ \tag{2.3}\label{2.3} \Omega_1(e_i,e_j)=\Omega_2(e_i,e_j)=\Omega_1 \left(e_i,(\overline{e_j})J \right)=\Omega_2 \left(e_i,(\overline{e_j})J \right)=0 \text{ for } i \not = j.\] \end{center}
\end{lemma}
\begin{proof}[Proof of Lemma \ref{diagon}] 
We proceed for a fixed $x \in M$. 

Take an orthonormal basis for $\Omega_1$ i.e. the basis $v_1,(\overline{v_1})J,...,v_n,(\overline{v_n})J$ such that (\ref{2.3}) is satisfied for $\Omega_1$ and in addition $\Omega_1(v_i,\overline{v_i}J)=1$ for $1 \leq i \leq n$. With its aid one is able to check that the endomorphism
\begin{center} $\widetilde{\Omega_2}: T^{1,0}_x M \longrightarrow T^{1,0}_x M $\end{center} 
defined by the relation 
\begin{center} $\Omega_2(v,\cdot)=\Omega_1 \left( \widetilde{\Omega_2}(v), \cdot \right)$ \end{center}
is actually well defined since $\widetilde{\Omega_2}(v) = \sum\limits_{1 \leq i \leq n} \left( \Omega_2(v,\overline{v_i}J) v_i - \Omega_2(v,v_i) \overline{v_i}J \right)$. 
We prove by induction that for any $1 \leq k \leq n$ there exist linearly independent vectors $e_1,(\overline{e_1})J,...,e_k,(\overline{e_k})J$ and complex numbers $\lambda_1,...,\lambda_k$ such that (\ref{2.3}) is satisfied and $\widetilde{\Omega_2}(e_i)=\lambda_i e_i$ for any $1 \leq i \leq k$.

For $k=1$ take any eigenvector $e_1$ for $\widetilde{\Omega_2}$, it is linearly independent of $\overline{e_1}J$ and (\ref{2.3}) is trivially satisfied. 

Assume that the claim holds for a fixed $1 \leq k<n$ and take a set of vectors like in the statement for $k$. Let us note that for any $u,v \in T^{1,0}_x M$, using only the q-reality of $\Omega_1$, $\Omega_2$ and the definition of $\widetilde{\Omega_2}$, we obtain
\begin{center}
$\Omega_2(\overline{v}J,u)= - \Omega_2 \left( \overline{v}J,\left( \overline{\overline{u}J} \right) J \right) = -(J \Omega_2)(\overline{v},\overline{\overline{u}J})=-\overline{\Omega_2}(\overline{v},\overline{\overline{u}J})=- \overline{\Omega_2(v,\overline{u}J)}=- \overline{\Omega_1(\widetilde{\Omega_2}(v),\overline{u}J)} = - \overline{\Omega_1} \left(\overline{\widetilde{\Omega_2}(v)},\overline{\overline{u}J} \right)=- (J \Omega_1) \left(\overline{\widetilde{\Omega_2}(v)},\overline{\overline{u}J} \right)= -\Omega_1 \left(\overline{\widetilde{\Omega_2}(v)}J,\overline{\overline{u}J}J \right) = \Omega_1 \left(\overline{\widetilde{\Omega_2}(v)}J,u \right)$
\end{center}
thus proving that 
\begin{center} $\widetilde{\Omega_2}(\overline{v}J)=\overline{\widetilde{\Omega_2}(v)}J$ for any $v \in T^{1,0}_x M$. \end{center}
Since $\widetilde{\Omega_2}(e_i)=\lambda_i e_i$, by the above, $\widetilde{\Omega_2}(\overline{e_i}J)=\overline{\lambda_i} \left(\overline{e_i}J \right)$ for $1 \leq i \leq k$. Consequently 
\begin{center} $\ker \Omega_1(e_i,\cdot) \subset  \ker \Omega_2(e_i,\cdot)$ and  $\ker \Omega_1(\overline{e_i}J,\cdot) \subset  \ker \Omega_2(\overline{e_i}J,\cdot)$ for $1 \leq i \leq k$.\end{center}
We introduce the following subspaces of $T_x^{1,0} M$ 
\begin{center} $V=span \{ e_1,(\overline{e_1})J,...,e_k,(\overline{e_k})J \}$, \\ 
$V'=\ker \Omega_1(e_1,\cdot) \cap \ker \Omega_1 \left( (\overline{e_1})J,\cdot \right) \cap ... \cap \ker \Omega_1(e_k,\cdot) \cap \ker \Omega_1 \left( (\overline{e_k})J,\cdot \right)$, \\
$V''=\ker \Omega_2(e_1,\cdot) \cap \ker \Omega_2 \left( (\overline{e_1})J,\cdot \right) \cap ... \cap \ker \Omega_2(e_k,\cdot) \cap \ker \Omega_2 \left( (\overline{e_k})J,\cdot \right)$.\end{center}
Note that $T_x^{1,0} M = V \oplus V'$ because $V \cap V' = \emptyset$ and $dim_\mathbb{C} V' \geq 2n-2k$. Let us also observe that $$\widetilde{\Omega_2}_{|V'}: V' \longrightarrow V' $$ since for $v \in V'$, by definition, $\widetilde{\Omega_2}(v)$ is such that $\Omega_2(v,\cdot)=\Omega_1 \left( \widetilde{\Omega_2}(v), \cdot \right)$ and $V' \subset V''$. Take $e_{k+1}$ to be any eigenvector for $\widetilde{\Omega_2}_{|V'}$. Since $e_{k+1} \in V'$ and $\Omega_1$ is q-real also $(\overline{e_{k+1}})J \in V'$. Finally due to the inclusion $V' \subset V''$ the linearly independent vectors $e_1,(\overline{e_1})J,...,e_{k+1},(\overline{e_{k+1}})J $ satisfy (\ref{2.3}) and thus all the required properties of the claim for $k+1$.
\end{proof}
\begin{remark}
A similar statement, Proposition 3.2, is contained in \cite{V10} and justified by saying that it follows from "a standard argument which gives simultaneous digitalization of two pseudo-Hermitian forms". We do not understand why this diagonalization is possible without assuming at least one of $\Omega_1$ or $\Omega_2$ being positive because in general two pseudo-Hermitian forms are diagonalizable simultaneously if at least one of them is positive.
\end{remark}
After normalization of $e_i$'s we may assume that 
\begin{center} 
$\Omega = e_1^* \wedge J^{-1} \left( \overline{e_1^*} \right) + ... + e_n^* \wedge J^{-1} \left( \overline{e_n^*} \right)$ \\ $\Omega_\phi = \phi_1 e_1^* \wedge J^{-1} \left( \overline{e_1^*} \right) + ... + \phi_n e_n^* \wedge J^{-1} \left( \overline{e_n^*} \right)$ for $\phi_i \geq 0$.
\end{center}
Let us decompose 
\begin{center}
$\beta=\sum\limits_{i=1}^n b_{2i-1}e_i^* + b_{2i} J^{-1}(\overline{e_i^*})$, $\partial \phi = \sum\limits_{i=1}^n a_{2i-1}e_i^* + a_{2i} J^{-1}(\overline{e_i^*})$ \\
then $\jpar \phi = J^{-1}( \overline{\partial} \phi) = J^{-1} \left( \overline{\partial \phi } \right)= \sum\limits_{i=1}^n - \overline{a_{2i}} e_i^* + \overline{a_{2i-1}} J^{-1}(\overline{e_i^*})$.
\end{center}
Since $b_i$'s are the coefficients of $\beta$ in a unitary basis they are uniformly bounded by $|\beta|_g$. One easily checks the equalities
\begin{center}
$\Omega_\phi^k \wedge \Omega^{n-k} = \frac{k!(n-k)!}{n!} \sum\limits_{1 \leq i_1 < ... < i_k \leq n} \phi_{i_1}...\phi_{i_k} \Omega^n$,\\
$\partial \phi \wedge \jpar \phi \wedge \Omega_\phi^k \wedge \Omega^{n-(k+1)} = \frac{k!(n-k-1)!}{n!} \sum\limits_{1 \leq i_1 < ... < i_k \leq n} \left( \sum\limits_{j \not \in \{i_1,...,i_k\}} |a_{2j-1}|^2 + |a_{2j}|^2 \right) \phi_{i_1}...\phi_{i_k} \Omega^n$, \\
$ \jpar \phi \wedge \beta \wedge \Omega_\phi^k \wedge \Omega^{n-(k+1)} = \frac{k!(n-k-1)!}{n!} \sum\limits_{1 \leq i_1 < ... < i_k \leq n} \left( \sum\limits_{j \not \in \{i_1,...,i_k\}} -\overline{a_{2j}}b_{2j} -\overline{a_{2j-1}} b_{2j-1} \right) \phi_{i_1}...\phi_{i_k} \Omega^n$.
\end{center}
Thus we see that it is enough to prove that there exists $B$ such that for any $0 \leq k<n$ and $\epsilon >0$
\begin{center}
\[ \begin{split}
& \sum\limits_{1 \leq i_1 < ... < i_k \leq n} \left( \sum\limits_{j \not \in \{i_1,...,i_k\}} |a_{2j}||b_{2j}| +|a_{2j-1}| |b_{2j-1}| \right) \phi_{i_1}...\phi_{i_k} \\
& \leq B \epsilon (n-k) \sum\limits_{1 \leq i_1 < ... < i_k \leq n} \phi_{i_1}...\phi_{i_k} + \frac{B}{\epsilon} \sum\limits_{1 \leq i_1 < ... < i_k \leq n} \left( \sum\limits_{j \not \in \{i_1,...,i_k\}} |a_{2j-1}|^2 + |a_{2j}|^2 \right) \phi_{i_1}...\phi_{i_k}. \end{split}\]
\end{center}
We have the string of inequalities following from the bound on $b_i$'s and the AM--GM inequality
\begin{center}
\[ \begin{split}  
& \sum\limits_{1 \leq i_1 < ... < i_k \leq n} \left( \sum\limits_{j \not \in \{i_1,...,i_k\}} |a_{2j}||b_{2j}| +|a_{2j-1}| |b_{2j-1}| \right) \phi_{i_1}...\phi_{i_k} \\
& \leq |\beta|_g \sum\limits_{1 \leq i_1 < ... < i_k \leq n} \left( \sum\limits_{j \not \in \{i_1,...,i_k\}} |a_{2j}|\phi_{i_1}...\phi_{i_k}  +|a_{2j-1}|\phi_{i_1}...\phi_{i_k}  \right) \\
& \leq \frac{|\beta|_g}{2} \sum\limits_{1 \leq i_1 < ... < i_k \leq n} \left( \sum\limits_{j \not \in \{i_1,...,i_k\}} 2 \epsilon \phi_{i_1}...\phi_{i_k} + \frac{|a_{2j}|^2\phi_{i_1}...\phi_{i_k} +|a_{2j-1}|^2 \phi_{i_1}...\phi_{i_k}}{\epsilon} \right) \end{split}\]
\end{center}
so we get that taking $B=|\beta|_g$ will do. 
\end{proof}
Having Lemma \ref{pointwise} established we are ready to deal with the term involving $\jpar \phi \wedge \beta \wedge \alpha \wedge \overline{\Omega^n}$ in the inequality (\ref{2.1}). 
\begin{lemma} \label{induk} There exist positive constants $C_1,...,C_{n}, \epsilon_1,...,\epsilon_{n}$ depending on the quantities listed in Lemma \ref{Cherrier} such that 
\begin{center} \[ \tag{2.4} \label{2.4} \begin{split}\frac p {2^i} \int\limits_{M}e^{-p\phi} \partial \phi \wedge \jpar \phi \wedge \alpha \wedge \overline{\Omega^n} \leq 
C_i \parallel e^{- \phi} \parallel_{L^{pr}}^p + \epsilon C_i \sum\limits_{k=1}^{n-i} \int\limits_{M}e^{-p\phi} \Omega_{\phi}^k \wedge \Omega^{n-k} \wedge \overline{\Omega^n} \end{split} \]  \end{center}
for all $i \in \{1,...,n\}$, $\epsilon \in (0 , \epsilon_{i} ]$ and $p \geq p_{i}( \epsilon )$ a positive number depending on $\epsilon$ and $i$. \end{lemma}
\begin{proof}[Proof of Lemma \ref{induk}] 
We show the claim by induction for a fixed $q>2n$. 

For the case $i=1$ let us note that from (\ref{2.2}) there exists a uniform positive constant $B$ such that for any $\epsilon > 0$ and $p>0$ 
\begin{center} \[ \begin{split} 
&-\int\limits_{M}e^{-p\phi} \jpar \phi \wedge \beta \wedge \alpha \wedge \overline{\Omega^n} = - \sum\limits_{k=0}^{n-1} \int\limits_{M}e^{-p\phi} \jpar \phi \wedge \beta \wedge \Omega_{\phi}^k \wedge \Omega^{n-1-k} \wedge \overline{\Omega^n} \\
& \leq  \sum\limits_{k=0}^{n-1} \left( \frac B \epsilon \int\limits_{M}e^{-p\phi} \partial \phi \wedge \jpar \phi \wedge \Omega_{\phi}^k \wedge \Omega^{n-1-k} \wedge \overline{\Omega^n} + B \epsilon \int\limits_{M}e^{-p\phi} \Omega_{\phi}^k \wedge \Omega^{n-k} \wedge \overline{\Omega^n} \right). \end{split}\] \end{center} 
We set $\epsilon_1=1$, then, by above, for any $\epsilon \leq \epsilon_1$ and  $p \geq p_1(\epsilon):= \frac {2B} \epsilon$ 
\begin{center} \[ \begin{split} 
& -\int\limits_{M}e^{-p\phi} \jpar \phi \wedge \beta \wedge \alpha \wedge \overline{\Omega^n} \\
& \leq \frac p 2 \int\limits_{M}e^{-p\phi} \partial \phi \wedge \jpar \phi \wedge \alpha \wedge \overline{\Omega^n} +B \int\limits_{M}e^{-p\phi}\Omega^{n} \wedge \overline{\Omega^n} + \epsilon B \sum\limits_{k=1}^{n-1} \int\limits_{M}e^{-p\phi} \Omega_{\phi}^k \wedge \Omega^{n-k} \wedge \overline{\Omega^n}. \end{split} \] \end{center}
This in turn, coupled with the inequality (\ref{2.1}) and H\"older's inequality, gives 
\begin{center} $\frac p 2 \int\limits_{M}e^{-p\phi} \partial \phi \wedge \jpar \phi \wedge \alpha \wedge \overline{\Omega^n} \leq C_1 \parallel e^{- \phi} \parallel_{L^{pr}}^p +  \epsilon C_1 \sum\limits_{k=1}^{n-1} \int\limits_{M}e^{-p\phi} \Omega_{\phi}^k \wedge \Omega^{n-k} \wedge \overline{\Omega^n}$\end{center}
proving the claim for $i=1$. 

For the inductive step suppose the claim holds for some fixed $1 \leq i<n$. To prove (\ref{2.4}) for $i+1$ we note that the LHS of (\ref{2.4}) for $i$ is twice the LHS of (\ref{2.4}) for $i+1$. Consequently it is enough to estimate the RHS of (\ref{2.4}) for $i$ by ones the LHS of (\ref{2.4}) for $i+1$ and the terms appearing on the RHS of (\ref{2.4}) for $i+1$. Note that since $\Omega_\phi = \Omega + \partial \jpar \phi$ we get
\begin{center} \[ \tag{2.5} \label{2.5} \begin{split} 
& \epsilon C_i \sum\limits_{k=1}^{n-i} \int\limits_{M}e^{-p\phi} \Omega_{\phi}^k \wedge \Omega^{n-k} \wedge \overline{\Omega^n} \\
& = \epsilon C_i \sum\limits_{k=1}^{n-i} \int\limits_{M}e^{-p\phi} \Omega_{\phi}^{k-1} \wedge \Omega^{n-(k-1)} \wedge \overline{\Omega^n}  \\ 
& + \epsilon C_i \sum\limits_{k=1}^{n-i} \int\limits_{M}e^{-p\phi} \partial \jpar \phi \wedge \Omega_{\phi}^{k-1} \wedge \Omega^{n-k} \wedge \overline{\Omega^n}, \end{split} \] \end{center}
because of the form of the RHS of (\ref{2.4}) for $i+1$ we only need to estimate the second summand. Applying Stokes' theorem and the fact that $\partial \overline{\Omega^n}= \beta \wedge \overline{\Omega^n}$ gives 
\begin{center} \[ \tag{2.6} \label{2.6} \begin{split} 
& \epsilon C_i \sum\limits_{k=1}^{n-i} \int\limits_{M}e^{-p\phi} \partial \jpar \phi \wedge \Omega_{\phi}^{k-1} \wedge \Omega^{n-k} \wedge \overline{\Omega^n}  \\
& = \epsilon p C_i \sum\limits_{k=1}^{n-i} \int\limits_{M}e^{-p\phi} \partial \phi \wedge \jpar \phi \wedge \Omega_{\phi}^{k-1} \wedge \Omega^{n-k} \wedge \overline{\Omega^n}  \\ 
& + \epsilon C_i \sum\limits_{k=1}^{n-i} \int\limits_{M}e^{-p\phi} \jpar \phi \wedge \beta \wedge \Omega_{\phi}^{k-1} \wedge \Omega^{n-k} \wedge \overline{\Omega^n}. \end{split} \] \end{center}
Below we bound both these summands. Let us set $\epsilon_{i+1}$ to be such that $\epsilon_{i+1} \leq \min\{\frac 1 {C_i 2^{i+2}}, \epsilon_i, 1\}$ then for any $\epsilon \in (0, \epsilon_{i+1}]$ and $p\geq p_i(\epsilon)$ 
\begin{center} \[ \tag{2.7} \label{2.7} \begin{split}  
& \epsilon p C_i \sum\limits_{k=1}^{n-i} \int\limits_{M}e^{-p\phi} \partial \phi \wedge \jpar \phi \wedge \Omega_{\phi}^{k-1} \wedge \Omega^{n-k} \wedge \overline{\Omega^n} \\ 
& \leq  \frac p {2^{i+2}} \sum\limits_{k=1}^{n} \int\limits_{M}e^{-p\phi} \partial \phi \wedge \jpar \phi \wedge \Omega_{\phi}^{k-1} \wedge \Omega^{n-k} \wedge \overline{\Omega^n} \\
& = \frac p {2^{i+2}} \int\limits_{M}e^{-p\phi} \partial \phi \wedge \jpar \phi \wedge \alpha \wedge \overline{\Omega^n}. \end{split} \] \end{center}
For any $\epsilon \in (0, \epsilon_{i+1}]$ we set $p_{i+1}(\epsilon)$ to be such that $p_{i+1}(\epsilon) \geq \max \{ p_i(\epsilon),2^{i+2}BC_i \}$ because then, again using firstly (\ref{2.2}), for $p \geq p_{i+1}(\epsilon)$  
\begin{center} \[ \tag{2.8} \label{2.8} \begin{split}  
& \epsilon C_i \sum\limits_{k=1}^{n-i} \int\limits_{M}e^{-p\phi} \jpar \phi \wedge \beta \wedge \Omega_{\phi}^{k-1} \wedge \Omega^{n-k} \wedge \overline{\Omega^n} \\
& \leq \epsilon C_i \sum\limits_{k=1}^{n-i} \frac B \epsilon \int\limits_{M}e^{-p\phi} \partial \phi \wedge \jpar \phi \wedge \Omega_{\phi}^{k-1} \wedge \Omega^{n-k} \wedge \overline{\Omega^n} \\ 
& + \epsilon C_i \sum\limits_{k=1}^{n-i} B \epsilon \int\limits_{M}e^{-p\phi} \Omega_{\phi}^{k-1} \wedge \Omega^{n-(k-1)} \wedge \overline{\Omega^n} \leq 
\frac{p}{2^{i+2}} \int\limits_{M}e^{-p\phi} \partial \phi \wedge \jpar \phi \wedge \alpha \wedge \overline{\Omega^n} \\ 
& + C_i B \int\limits_{M}e^{-p\phi} \Omega^n \wedge \overline{\Omega^n} + \epsilon C_i B  \sum\limits_{k=1}^{n-(i+1)} \int\limits_{M}e^{-p\phi} \Omega_{\phi}^{k} \wedge \Omega^{n-k} \wedge \overline{\Omega^n}.  \end{split} \] \end{center}
Note that for $\epsilon \in (0, \epsilon_{i+1}]$ and $p \geq p_{i+1}(\epsilon)$, from (\ref{2.4}),
\begin{center} $2 \frac p {2^{i+1}} \int\limits_{M}e^{-p\phi} \partial \phi \wedge \jpar \phi \wedge \alpha \wedge \overline{\Omega^n} \leq 
C_i \parallel e^{- \phi} \parallel_{L^{pr}}^p + \epsilon C_i \sum\limits_{k=1}^{n-i} \int\limits_{M}e^{-p\phi} \Omega_{\phi}^k \wedge \Omega^{n-k} \wedge \overline{\Omega^n}.$\end{center} 
By (\ref{2.5}) the RHS of the above inequality equals to
\begin{center} $C_i \parallel e^{- \phi} \parallel_{L^{pr}}^p + \epsilon C_i \sum\limits_{k=1}^{n-i} \int\limits_{M}e^{-p\phi} \Omega_{\phi}^{k-1} \wedge \Omega^{n-(k-1)} \wedge \overline{\Omega^n}  +  \epsilon C_i \sum\limits_{k=1}^{n-i} \int\limits_{M}e^{-p\phi} \partial \jpar \phi \wedge \Omega_{\phi}^{k-1} \wedge \Omega^{n-k} \wedge \overline{\Omega^n}.$ \end{center}
Then by rewriting the second summand and applying (\ref{2.6}) for the last one the above expression becomes
\begin{center} \[ \begin{split} 
& C_i \parallel e^{- \phi} \parallel_{L^{pr}}^p + \epsilon C_i \int\limits_{M}e^{-p\phi} \Omega^n \wedge \overline{\Omega^n} + \epsilon C_i \sum\limits_{k=1}^{n-(i+1)} \int\limits_{M}e^{-p\phi} \Omega_{\phi}^{k} \wedge \Omega^{n-k} \wedge \overline{\Omega^n} \\ & + \epsilon p C_i \sum\limits_{k=1}^{n-i} \int\limits_{M}e^{-p\phi} \partial \phi \wedge \jpar \phi \wedge \Omega_{\phi}^{k-1} \wedge \Omega^{n-k} \wedge \overline{\Omega^n} \\
& + \epsilon C_i \sum\limits_{k=1}^{n-i} \int\limits_{M}e^{-p\phi} \jpar \phi \wedge \beta \wedge \Omega_{\phi}^{k-1} \wedge \Omega^{n-k} \wedge \overline{\Omega^n}. \end{split} \] \end{center}
Applying (\ref{2.7}) for the last but one summand, (\ref{2.8}) for the last one and H\"older's inequality to bound $\parallel e^{- \phi} \parallel_{L^{p}}^p$ by $\parallel e^{- \phi} \parallel_{L^{pr}}^p$ shows that this quantity is estimated by 
\begin{center}$C_{i+1} \parallel e^{- \phi} \parallel_{L^{pr}}^p + \epsilon C_{i+1} \sum\limits_{k=1}^{n-(i+1)} \int\limits_{M}e^{-p\phi} \Omega_{\phi}^k \wedge \Omega^{n-k} \wedge \overline{\Omega^n} + 2 \frac p {2^{i+2}} \int\limits_{M}e^{-p\phi} \partial \phi \wedge \jpar \phi \wedge \alpha \wedge \overline{\Omega^n}$, \end{center}
for a constant $C_{i+1}$ depending on $B$ and $C_i$. We obtain (\ref{2.4}) for $i+1$ and this finishes the proof of the inductive step.
\end{proof}
The proof of the main result, the Cherrier type inequality, is now finished by taking for a given $q>2n$ in Lemma \ref{induk}, $i=n$, $\epsilon=\epsilon_n$ and $p_0=p_n(\epsilon_n)$ because then for any $p \geq p_0$
\begin{center} \[ \begin{split} 
& \int\limits_{M} |\partial e^{-\frac p 2 \phi }|^{2}_{g} \left(\Omega \wedge \overline{\Omega} \right)^n =
n  \int\limits_{M} \partial e^{-\frac p 2 \phi } \wedge \jpar e^{-\frac p 2 \phi } \wedge \Omega^{n-1} \wedge \overline{\Omega^n} =
\frac{n p^2}{4} \int\limits_{M} e^{-p \phi} \partial \phi  \wedge \jpar  \phi  \wedge \Omega^{n-1} \wedge \overline{\Omega^n} \\
& \leq p C \left( \frac {p}{2^n} \int\limits_{M} e^{-p \phi} \partial \phi  \wedge \jpar  \phi  \wedge \alpha \wedge \overline{\Omega^n} \right) \leq
p C \parallel e^{- \phi} \parallel_{L^{pr}}^p. \end{split} \] \end{center} 
\end{proof}

\subsection{The $C^0$ estimate}
\begin{lemma}\label{prep1} There exist positive constants $C$ and $s_0$, depending on the quantities listed in Lemma \ref{Cherrier}, such that for any solution of (\ref{1.1}) \begin{center} $e^{-s_0 \inf\limits_{M} \phi} \leq e^{C} \int\limits_{M} e^{-s_0 \phi} \left( \Omega \wedge \overline{\Omega} \right)^n$. \end{center} \end{lemma}
\begin{proof}[Proof of Lemma \ref{prep1}]
From the Sobolev inequality for $(M,g)$, the fact that $\Omega^n \wedge \overline{\Omega^n}$ is uniformly comparable with the Riemannian volume element, Lemma \ref{Cherrier} and the H\"older inequality we obtain
\begin{center} $\left( \int\limits_{M} e^{-p  \phi \gamma}\Omega^n \wedge \overline{\Omega^n} \right)^{\frac 1 \gamma} \leq C \left( \int\limits_{M} | \nabla e^{-\frac p 2 \phi }|_g^2 \Omega^n \wedge \overline{\Omega^n} + \int\limits_{M} e^{ -p \phi } \Omega^n \wedge \overline{\Omega^n} \right) \leq pC \parallel e^{- \phi} \parallel_{L^{pr}}^p $ \end{center}
for $r<\gamma:=\frac {2n} {2n-1}$ H\"older's conjugate of $q$, a uniform constant $C$ and any $p \geq p_0$. This is equivalent to 
\begin{center} $\parallel e^{- \phi} \parallel_{L^{(p \frac \gamma r)r}}=\parallel e^{-\phi}\parallel_{L^{p\gamma}} \leq (pC)^{\frac 1 p} \parallel e^{-\phi}\parallel_{L^{pr}}$.\end{center}
The iteration of the last inequality for $p_0$, $p_0 \frac \gamma r$, $p_0 \left(\frac \gamma r \right)^2$ and so on gives 
\begin{center} $\sup\limits_{M} e^{-\phi} \leq C \parallel e^{-\phi}\parallel_{L^{p_0r}}$ \end{center}
hence we can take $s_0:=p_0r$ since then
\begin{center} $e^{-s_0 \inf\limits_{M} \phi} \leq C \int\limits_{M} e^{-s_0\phi} \Omega^n \wedge \overline{\Omega^n}$. \end{center}
\end{proof}

\begin{lemma}\label{prep2}{\cite{TW10a}} There exist positive constants $C_1$, $C_2$ such that for any solution of (\ref{1.1}) \begin{center} $\int\limits_{ \{\phi \leq \inf\limits_{M} \phi +C_1 \}} \left( \Omega \wedge \overline{\Omega} \right)^n \geq C_2$. \end{center} \end{lemma}
\begin{proof}[Proof of Lemma \ref{prep2}]
Having Lemma \ref{prep1}, the proof is exactly as in \cite{TW10a}. The normalization of the volume element they use is purely for computational convenience.
\end{proof}

\begin{proof}[Proof of Theorem \ref{main}]
As it has been said in order to finish the proof one needs at least an $L^1$ bound on $\phi$. This estimate was shown in Proposition 2.3 in \cite{AS13}. The proof is now finished by noting that either
\begin{center} $\inf\limits_{M} \phi +C_1 \geq 0$ giving $- \inf\limits_{M} \phi \leq C_1$ \end{center}
or
\begin{center}  
$C_3 \geq \: \parallel \phi \parallel_{L^1} \geq \int\limits_{ \{\phi \leq \inf\limits_{M} \phi +C_1 \}} |\phi| \left( \Omega \wedge \overline{\Omega} \right)^n \geq C_2 \left( -\inf\limits_{M} \phi -C_1 \right)$.
\end{center}
This gives a uniform constant $C=\max \{C_1, \frac{C_3}{C_2}+C_1 \}=\frac{C_3}{C_2} + C_1$ for which
\begin{center}$-\inf\limits_{M} \phi \leq C$.
\end{center}

\end{proof}

\noindent FACULTY OF MATHEMATICS AND COMPUTER SCIENCE\\
OF JAGIELLONIAN UNIVERSITY \\
\L OJASIEWICZA  6 \\
30-348, KRAK\'OW \\
POLAND \\
\textit{E-mail address:} Marcin.sroka@im.uj.edu.pl

\end{document}